\documentclass[12pt]{amsart}
\usepackage[letterpaper,top=3cm,bottom=3cm,left=3cm,right=3cm,marginparwidth=1.75cm]{geometry}
\subjclass[2020]{42B25, 43A80}
\keywords{Maximal functions, generalized H-type groups, uniform volume estimates}
\usepackage{amsmath}
\usepackage{esint}
\usepackage{amssymb}
\usepackage{cite}
\usepackage{esint}
\usepackage{bbm}
\usepackage{amsthm}
\usepackage{amsfonts}
\usepackage{mathrsfs}
\usepackage{amscd}
\usepackage{amssymb} 
\usepackage{cancel}
\usepackage{latexsym}
\usepackage{eufrak}
\usepackage{euscript}
\usepackage{epsfig}
\usepackage{graphicx}
\usepackage{appendix}
\usepackage[colorlinks=true, allcolors=blue]{hyperref}
\usepackage{tcolorbox}
\numberwithin{equation}{section}
\newtheorem{theorem}{Theorem}[section]

\newtheorem{lemma}[theorem]{Lemma}
\newtheorem{corollary}[theorem]{Corollary}
\newtheorem{assumption}[theorem]{Assumption}

\newtheorem{remark}[theorem]{Remark}

\newtheorem{prop}[theorem]{Proposition}

\newcommand{\qe}{\end{equation}}
\newcommand{\R}{{\mathbb R}}
\newcommand{\HH}{{\mathbb G}}
\newcommand{\N}{{\mathbb N}}

\newcommand{\sS}{\mathrm S}
\newcommand{\G}{\mathbb{G}(2n,m, \mathbb{U}, \mathbb{W})}
\newcommand{\sumj}{\sum_{j=1}^{\ell}}
\newcommand{\xf}{x_{(1)}}
\newcommand{\xj}{x_{(j)}}
\newcommand{\xl}{x_{(\ell)}}

\newcommand{\A}{{\mathbb A}}

\newcommand{\CCC}{\mathcal{C}_{\HH}}

\newcommand{\lal}{ \lambda^T \mathbb{A} \lambda}

\newcommand{\Tg}{\mathfrak{T}_{x, h}}
\def\Sa{\Sigma_1(R)}
\def\Sb{\Sigma_2(R)}

\def\SBb{\Sigma_2(1)}
\def\DB{d_{\mathbb{B}}}
\def\Fx{\mathbf{F}_x}
\def\K{\varrho}
\def\wX{\widetilde{\mathrm{X}}}
\def\GG{\mathfrak{G}}
\def\CG{\mathcal{G}}
\def\TA{\mathfrak{W}_{x, h}}
\def\TB{\mathfrak{V}_{x, h}}
\def\MC{\mathscr{C}_{x,h}}
\def\rr{\mathrm{r}}

\begin{document}

\title[Uniform Volume estimates]{Uniform volume estimates and maximal functions on generalized Heisenberg-type groups}
\author{Cheng Bi,  Hong-Quan Li}

\maketitle

\begin{abstract}
	On generalized Heisenberg-type groups  $\G$, we give uniform volume estimates for the ball defined by a large class of Carnot-Carath\'{e}odory distances, and establish weak (1, 1)  $O(C^m \, n)$-estimates for associated centered Hardy-Littlewood maximal functions, extending the results in \cite{BLZ25}. As a by-product, we establish uniformly volume doubling property on Heisenberg groups for a class of left-invariant Riemannian metrics.
\end{abstract}

\section{Introduction}
	Recently, the phenomenon of uniformly volume doubling was studied by N. Eldredge, M. Gordina and L. Saloff-Coste \cite{EGS18}, \cite{EGS24}.  More precisely, let $\mathbb{X}$ be a finite-dimensional connected real Lie group. For a left-invariant Riemannian metric $\mathbf{g}$ on $\mathbb{X}$, denote by $ d_\mathbf{g}$ the induced Riemannian distance, by $ \nu_\mathbf{g}$ the corresponding Riemannian measure and by $B_{\mathbf{g}}(z,r)$ the open ball centered at $z \in \mathbb{X}$ with radius $r > 0$. The doubling constant $ D_\mathbf{g}$ is defined by
	\begin{equation*}
		D_\mathbf{g}  :=  \sup_{z \in \mathbb{X}, \, r>0} \frac{\nu_\mathbf{g} (B_{\mathbf{g}}(z , 2r))}{ \nu_\mathbf{g}(B_{\mathbf{g}}(z ,r))}.
	\end{equation*}
	  We say $\mathbb{X}$ is {\emph{uniformly doubling}} if
	$D(\mathbb{X}) := \sup_{ \mathbf{g} \in \mathcal{L}_{\mathbb{X}} } D_\mathbf{g} <\infty,$ where $\mathcal{L}_{\mathbb{X}}$ is the set of all left-invariant Riemannian metrics on $\mathbb{X}$. 
    As noted in \cite{EGS18}, \cite{EGS24}, this property is particularly significant, as it implies a series of analytical consequences such as a uniform Poincar\'{e} inequality \cite[Theorem 8.2]{EGS18} and uniform two-sided heat kernel estimates \cite[Theorem 8.10]{EGS18}. 
    
    As noted in \cite[p. 1324]{EGS18},  the Euclidean spaces and the tori are uniformly doubling. Besides, the authors of \cite{EGS18} conjectured that every connected real compact Lie group is uniformly doubling. They proved this conjecture for the three-dimensional special unitary group SU(2) \cite[Theorem 1.2]{EGS18} as well as the quotient groups of SU(2)$\times \R^n$  (including non-compact ones) \cite[Theorem 1.1]{EGS24}. However, to the best of our knowledge, there {\color{blue}are} no other nontrivial examples  for which the conjecture holds. 
       
    A natural question is whether the famous Heisenberg groups are uniformly doubling. These groups are non-compact and do not admit uniform lower and upper bounds on the Ricci curvatures (see e.g. \cite{M76}) even within the class of the left-invariant Riemmannian metrics considered in this work. Therefore, the celebrated Bishop-Gromov volume comparison theorem is not applicable. In this paper, we provide a partial answer to this question by establishing the uniform doubling property for very natural left-invariant Riemannian metrics. Indeed, in the broader setting of generalized H-type groups (see Section \ref{ssec-geneH}), we derive explicit formulas describing the volume growth of balls defined by a large family of Carnot-Carath\'{e}odory distances; see Theorem \ref{thm-voll} below. 

    With volume estimates in hand, we proceed to investigate  an extension of Stein-Str\"{o}mberg's $O(n)$ bound \cite{SS83} on the centered Hardy-Littlewood maximal functions. Our results are Theorems \ref{thm-main1} and \ref{thm-2}, which also extend the corresponding results in \cite{BLZ25}. For other related works, we refer to \cite{Li09,Li13,LL12,LQ14, NT10}.  
    
	This article is organized as follows. In Section \ref{sec-nota}, we recall some basic facts and state our main results. In Section \ref{sec-vol}, we derive uniform volume estimates. As an application, we study the weak (1, 1) estimates for the associated maximal functions in Section \ref{sec-app}.
		
\section{Settings and Statement of results}\label{sec-nota}

\subsection{Notation} 

Throughout this paper, we use $|\cdot|$ and $\cdot$ to denote the length and the inner product on Euclidean spaces, respectively. By a slight abuse of notation, we will also use $|E|$ to denote the Lebesgue measure (or equivalently the Haar measure) of a measurable set $E$ and $\cdot$ to denote the group multiplication. Furthermore, we consider vectors in Euclidean spaces as column vectors, and denote by $\A^T$ the transpose of a matrix $\A$, by $B_{\R^k}(0, r)$ the usual Euclidean ball centered at the origin with radius $r > 0$, by
$\langle \cdot , \cdot \rangle$ the inner product  in $\mathbb{R}^{2n}$. Moreover, we write $\hat{\lambda} = \lambda/|\lambda|$ if $\lambda\in \R^k\backslash\{0\}$. We denote by $\nabla$ or  $\nabla_\lambda$ with $\lambda\in\R^k$ the usual gradient, and Hess$_\lambda$ the Hessian matrix.

\subsection{Generalized H-type groups}\label{ssec-geneH}

Let $n, m \in \mathbb{N}_+ =\{ 1,2,\cdots\}$ and $\mathbb{U} = \{ U^{(1)},\cdots, U^{(m)} \} $ be an $m$-tuple of linearly independent $(2n) \times (2n)$ skew-symmetric real matrices. Let $\mathbb{W}$ be a positive definite real matrix such that
\begin{equation*}
U(\lambda) \, U(\lambda') + U(\lambda') \, U(\lambda) = 2 \lambda \cdot \lambda' \, \mathbb{W}^2, \quad \forall \, \lambda,\lambda' \in \R^m,
\end{equation*} 
where $ U(\lambda) = \mathrm{i} \sum_{l=1}^m \lambda_l \, U^{(l)}$ for $\lambda =(\lambda_1,\cdots,\lambda_m).$ 
The generalized H-type groups $\HH = \G$ can be identical to $\R^{2n} \times \R^m$ endowed with the group law 
\begin{gather*}
(x, t) \cdot (x', t') := \left( x + x', \ t + t' + \frac{1}{2} \langle \mathbb{U} x , x' \rangle _{\star} \right),  \   x, x' \in  \R^{2 n}, \ t, t' \in \R^m,  \\
\mbox{with} \ \langle \mathbb{U} x, x' \rangle _{\star} = \left( \langle {U}^{(1)} x, x' \rangle, \, \cdots,\, \langle {U}^{(m)} x, x' \rangle \right) \in \R^m.
\end{gather*}

Note that the classical H-type group $\mathbb{H}(2n,m)$ is the special case where $\mathbb{W} = \mathbb{I}_{2n}$, and the generalized Heisenberg group corresponds to the case $m=1$. 

Set $ U^{ (l) } = \left( U^{ (l) }_{ \iota,q } \right)_{ 1 \leq \iota, q \leq 2n}$ for $1 \leq l \leq m$. The corresponding left invariant vector fields  (namely the canonical basis)  are defined by
\begin{equation*}
\mathrm{ X }_{ q } := \frac{ \partial }{\partial x_{q}} + \frac{1}{2} \sum_{ l = 1 }^{m} \left( \sum_{ \iota = 1 }^{ 2n } U^{ (l) }_{q , \iota} 
x_{\iota}  \, \frac{ \partial }{ \partial  t_{ l }} \right), \quad 1 \le q \le 2 n.
\end{equation*}
The canonical sub-Laplacian is $\Delta :=  \sum_{ q=1 }^{ 2n } \mathrm{ X }_{ q }^{ 2 }$.
We shall also consider the following general sub-Laplacian
\begin{equation*}
\Delta_{\mathbb{B}} = \Delta_{\G, \mathbb{B}}:= \Delta + \sum_{i=1}^m \left( \sum_{l=1}^m \mathbb{B}_{i,l} \mathrm{T}_l \right)^2, \quad \mbox{with} \quad \mathrm{T}_l = \frac{\partial}{\partial t_l } \ (1 \le l \le m), 
\end{equation*}
where $\mathbb{B} = ( \mathbb{B}_{i,l} )$ is an $m\times m$ real matrix. Note that the cases $\mathbb{B}=0$ and $\mathbb{B}=\mathbb{I}_m$ correspond to the canonical sub-Laplacian and the full Laplacian, respectively.

Let $g = (x, t) \in \HH$ and $o = (0, 0)$ be the identity element of $\HH$. The Haar measure $dg$ on $\HH$ coincides with the $(2n+m)$-dimensional Lebesgue measure. Denote by $d_{\mathbb{B}} = d_{\mathbb{W}, \mathbb{B}}$ the left-invariant Carnot-Carath\'{e}odory distance associated to $\Delta_\mathbb{B}$, cf. e.g. \cite[III 4]{VSC92}. By the left-translation invariance, we set in the sequel $d_{\mathbb{B}}(g) := d_{\mathbb{B}}(g, o)$. 

\subsection{Explicit expression of $\DB(g)$}

 Let $\ell \in \N_+.$ Up to an orthogonal transformation w.r.t $x$, we may assume
  \begin{equation}\label{equ-W}
   \mathbb{W} = \mbox{diag} \left\{ a_1  \mathbb{I}_{2k_1}, \, \cdots, \, a_{\ell} \, \mathbb{I}_{2k_{\ell}} \right\}, \  \mbox{with} \  0<a_1<\cdots< a_{\ell} , \  \sum_{j=1}^\ell k_j = n,\ k_j \in \N_+.
 \end{equation}
The case of $\ell =1$ corresponds to the classical H-type group, the result for which can be found in \cite{BLZ25}.  In the rest of this paper, we always assume $\ell\ge 2$.
It follows from the Radon-Hurwitz number that:
 \begin{equation}\label{equ-m}
  m \le 2 \min_{1\le j \le \ell }  \log_2 (4k_j), \qquad m + 1 \le  \min_{1\le j \le \ell } 2 k_j,
 \end{equation}
which will be used repeatedly. Similarly, we may suppose in addition:
\begin{align}
    \mathbb{A} := \mathbb{B}^T \, \mathbb{B} = \mbox{diag}\{ b_1, \cdots, b_m\} \quad \mbox{with all $b_j \ge 0$.}
\end{align}

In view of \eqref{equ-W}, we write in the sequel
\begin{equation*}
x=\left( \xf, \cdots, \xl \right) \quad \mbox{with} \quad \xj \in \R^{2k_j}, \quad  j = 1,\cdots, \ell.
\end{equation*}

As in \cite{BGG96} and \cite[chapter 10]{CCFI}, the convolution kernel of $e^{h \Delta_\mathbb{B}}$ ($h > 0$) is given by:
\begin{equation}
p_h(g) = \frac{h^{-n-m}}{(2\pi)^m (4\pi)^n} \int_{\R^m} \mathbf{V}(\lambda)\exp\left(-\frac{\Phi(g;\lambda)}{4h}\right)\,d\lambda,
\end{equation}
where
\begin{equation*}
\mathbf{V}(\lambda) = \prod_{j=1}^{\ell} \left(\frac{a_j |\lambda|}{\sinh(a_j |\lambda|)}\right)^{k_j},
\ \Phi(g;\lambda) = \sumj |\xj|^2 a_j |\lambda| \coth(a_j |\lambda|)-4 \mathrm{i} t\cdot \lambda  + 4 \lambda^T \mathbb{A} \lambda.
\end{equation*}

Set for $-\pi < r < \pi$ (cf. e.g. \cite[\S~3]{Li21})
\begin{align} \label{MU}
\mu(r) := -\frac{d}{d r} ( r \, \cot{r} ) = \frac{2r - \sin(2r)}{2\sin^2 r} = 2 \, r \sum_{j = 1}^{+\infty} \Big( (j \, \pi)^2 - r^2 \Big)^{-2}, \quad  \widetilde{\mu} (r) := \frac{\mu(r)}{4 r}.
\end{align}
The former is a strictly increasing diffeomorphism from $(-\pi,\pi)$ onto $\R$. 

The following result can be readily deduced from \cite{Li21}:
\begin{theorem}[\cite{Li21}]\label{thm-Fx}
(1) Let $\Omega_* := B_{\R^m}\left(0,\pi/a_\ell \right)$. Then we have:   
\begin{equation}\label{equ-dB}
d_{\mathbb{B}}(g)^2  =  \sup_{\tau\in\Omega_* } \phi(g;\tau),  \quad \phi(g;\lambda) :=  \sumj  (a_j |\lambda|) \cot{(a_j |\lambda|)} \, |\xj|^2 + 4 t \cdot \lambda - 4 \lambda^T \A \lambda.
\end{equation}

(2) Let $x\in\mathbb{R}^{2n}$ with $x_{(\ell)} \neq 0$. Then the following map $\Fx$ is a $C^{\infty}$-diffeomorphism from $\Omega_*$ to $\R^m$,
\begin{equation}\label{equ-Atheta}
t_l = \theta_l \left(  \sumj a_j^2 \, \widetilde{\mu} (a_j|\theta|) \, |\xj|^2  + 2 b_l \right), \quad l=1,\cdots,m.
\end{equation}
Furthermore, it holds for any $\theta = (\theta_1, \ldots, \theta_m) \in \Omega_*$ that
\begin{align}\label{equ-dBB}
\DB(x, \Fx(\theta))^2 = \sum_{j = 1}^{\ell} \left( \frac{a_j \, |\theta|}{\sin{(a_j \, |\theta|)}} \right)^2 \, |x_{(j)}|^2 + 4 \sum_{k = 1}^m b_k \theta_k^2.
\end{align}
\end{theorem}

In what follows, we may assume that $x_{(\ell)} \neq 0$. 

\subsection{Statement of the results}\label{statement}

In the rest of the paper, the symbols $C, c$ will be used to denote various positive constants independent of $\mathbb{G}(2n,m,\mathbb{U},\mathbb{W})$ and $\mathbb{B}$,  which may vary from one line to the next. For non-negative functions $f_1$ and $f_2$, we write $f_1 \sim f_2$ if there exists  $C>0$ such that $ C^{-1} f_2 \leq f_1 \leq C f_2$. Similarly, $f_1 \lesssim f_2$ (resp. $f_1 \gtrsim f_2$) if $f_1 \le C f_2$ (resp. $f_1 \ge C f_2$). And $f_1 \ll f_2$ (resp. $f_1 \gg f_2$) if $f_1 \lesssim f_2$ (resp. $f_1 \gtrsim f_2$) with the constant $C$ small (resp. large) enough. Furthermore, for a complex-valued function $w$, the notation $w = O(f_1)$ means $|w| \le C f_1$. 

We introduce the following constants which will be used throughout this paper.
\begin{equation}\label{equ-KCH}
K := \sumj a_j^2 \, k_j =\frac{1}{2} \, \mbox{Tr}(\mathbb{W}^2)  \quad\mbox{and}\quad  \CCC :=  \frac{K}{n} \, (< a_\ell^2), 
\end{equation}
where $\mbox{Tr}(\mathbb{W}^2) $ is the trace of $\mathbb{W}^2$. Let $B_{\mathbb{B}}(g,r)$ be the ball defined by $d_\mathbb{B}$ centered at $g$ with radius $r> 0$. By left invariance, the volume $|B_{\mathbb{B}}(g,r)|= |B_{\mathbb{B}}(o,r)|$ for all $g\in\HH$. We give its uniform volume estimate.

\begin{theorem}\label{thm-voll}
We have uniformly in $r> 0$ that
\begin{align}\label{equ-con}
	\begin{split}
		|B_{\mathbb{B}} (o,r)| \sim |B_{\R^{2n+m}}(0,r)| \ \det \left( \frac{r^2 \, \CCC}{12}  \, \mathbb{I}_m + \mathbb{B}^T \mathbb{B} \right)^{\frac{1}{2}}.
	\end{split}
\end{align}
Consequently, 
\begin{equation}\label{equ-uniformg}
	\sup_{g\in \HH, \, r>0,\, \mathbb{B}} \frac{|B_{\mathbb{B}}(g,2r)|}{|B_{\mathbb{B}}(g,r)|} \lesssim 2^{2(n+m)}.
\end{equation}
\end{theorem}

\begin{remark}
Note that $2(n+m)$ is exactly the dimension at infinity of $\mathbb{G}$. 
\end{remark}

As a direct consequence, we obtain:

\begin{corollary}\label{cor-H}
    In the setting of Heisenberg groups $\mathbb{H}(2n, 1)$, the uniformly volume doubling property is valid for the following natural Laplacians:
    \[
    \sum_{j = 1}^{2 n} \wX_j^2 + c^2 \frac{\partial^2}{\partial t^2}, \quad \mbox{with $c > 0$ and $\{ \wX_j \}_{j = 1}^{2 n}$ a basis of $\mathrm{span} \{ \mathrm{ X }_1, \ldots, \mathrm{ X }_{2 n}\}$.} 
    \]
\end{corollary}

\begin{remark}
We will show in a future work that all step-two Carnot groups of corank $1$ admit the uniformly doubling property. In particular, the extra condition ``natural Laplacians'' in Corollary \ref{cor-H} can be removed. 
\end{remark}

Next, we consider the weak $(1, 1)$ inequality for the centered Hardy-Littlewood maximal function $M_\mathbb{B}$ associated to $d_{\mathbb{B}}$, which is defined by
    \begin{align}\label{max}
    	M_{\mathbb{B}} f(z) = \sup_{r>0} \, \fint_{ B_{\mathbb{B}}(z,r) } 
    	|f(g)| \,  dg,\quad z \in \mathbb{G},
    \end{align}
    where the integral with a stick denotes the integral average of $|f|$ over $B_\mathbb{B}(z, r)$.

    The following technical condition will be used in the study of asymptotics of Poisson kernel on $\mathbb{G}$; see Proposition \ref{prop-poi} and Remark \ref{remark-pi} below.
\begin{assumption}\label{assumption}
There exists $\eta_0 \in (0,1] $ such that $k_{\ell}/ n \, \ge \,\eta_0 $. 
\end{assumption}

\begin{theorem}\label{thm-main1}
If Assumption \ref{assumption} holds, then there exists a constant $C>0$,  depending only on $\eta_0$ such that
\begin{equation*}
	\| M_{\mathbb{B}} \|_{L^1 \longrightarrow L^{1,\infty}} \le C  \left(\frac{3 \, a_\ell^2 }{2 \, \CCC } 
	\right)^{\frac{m}{2}} n.
\end{equation*}
\end{theorem}

In the special case of H-type groups, where $\CCC= a_\ell =1$ and $\eta_0 = 1$,
this result extends \cite[Theorem 5.1]{BLZ25}. 

It follows from \eqref{equ-dB} that $d_{\mathbb{W}, \mathbb{B}}(x, t)^2 = d_{\mathbb{W}/a_{\ell}, \mathbb{B}/a_{\ell}}(x, t/a_{\ell})^2$. In what follows, up to a scaling w.r.t. $t$, we may assume $a_{\ell} = 1$. \\ 

\section{Proof of Theorem \ref{thm-voll}}\label{sec-vol}

The goal of this section is to establish volume estimates \eqref{equ-con}, whose strategy follows the spirit of \cite[Lemma~5.3]{BLZ25}. In our setting, however, some significant modifications are necessary due to the lack of symmetry w.r.t. $x$. Specifically, in computing the main contribution, we are faced with a more complicated integral estimate after making use of the diffeomorphism defined by \eqref{equ-Atheta}. A new inequality (Lemma \ref{lem-WW} below) will be applied repeatedly. For the remainder, we use the method of decomposition in annuli to deal with the singularity of the integrand. Let us begin with:

\subsection{Auxiliary estimates}
We present several elementary inequalities based on the Stirling's formula (see \eqref{St} below). 

Let $B_R := B_{\R^{2n}}(0,R)$ for $R>0$, and $\mathbb{S}^{k-1}$ denote the unit sphere in $\R^k$. Recall that its area is $|\mathbb{S}^{k-1}| = k \, |B_{\R^k} (0,1)|= 2\pi^{\frac{k}{2}} \, {\Gamma\left( \frac{k}{2} \right)}^{-1}$, where $\Gamma$ denotes the usual gamma function.
\begin{lemma}\label{lem-a}
(i) We have, uniformly in $1\le j\le \ell$ and $0\le u \le 100 \min_{1\le j\le \ell} k_j^{1/2}$, that
     \[
    \fint_{B_1} \, |\xj|^{2u} \, dx \sim  \left(\frac{k_j}{n}\right)^u.
     \]

(ii)
 Let $(n, m)$ be as in \eqref{equ-W}-\eqref{equ-m}, then 
 it holds, uniformly in $0\le v \le 10 m$, that
        \begin{equation}\label{beta}
            \int_{ ( 1+ \frac{2m}{n} )^{-\frac{1}{2}} }^1 \, (1-r^2)^{\frac{m}{2}} r^{2n+v-1} \, dr \sim \int_{ 0 }^1 \, (1-r^2)^{\frac{m}{2}} \, r^{2n+v-1} \, dr \sim \frac{\Gamma(\frac{m}{2}+1)}{n^{\frac{m}{2}+1}}.
        \end{equation}        
\end{lemma}

\begin{proof}
    For (i), using polar coordinates in $\R^{2n} = \mathbb{R}^{2k_j} \times \mathbb{R}^{2(n-k_j)}$, we get
    \begin{equation*}
        \fint_{B_1} \, |\xj|^{2u} \, dx = \frac{|\mathbb{S}^{2k_j-1}| \cdot |\mathbb{S}^{2(n-k_j)-1}|}{|B_1|} \int_{\mathop{r^2 + s^2 <1}\limits_{r>0,s>0}} r^{2(u+k_j)-1} s^{2(n-k_j)-1} \,d r \, d s.
    \end{equation*}
    Notice that for $\alpha>-1$ and $\beta>-1$ (cf. e.g. \cite[\S~4.635]{GR15}),
    \begin{equation}\label{Gamma}
        \int_{\mathop{r^2 + s^2 <1}\limits_{r>0,s>0}} r^\alpha s^\beta \,drds = \frac{1}{\beta+1} \int_0^1 r^\alpha \, (1-r^2)^{\frac{\beta+1}{2}} \, dr =  \frac{1}{4} \frac{\Gamma(\frac{\alpha+1}{2}) \, \Gamma(\frac{\beta+1}{2})}{\Gamma(\frac{\alpha+\beta}{2}+2)}.
    \end{equation}

Then we have
    \begin{equation*}
         \fint_{B_1} \, |\xj|^{2u} \, dx = 
         \frac{\Gamma(u+k_j)}{\Gamma(k_j)} \frac{\Gamma(n+1)}{\Gamma(u+n+1)}.
    \end{equation*}
    Now we use the Stirling's formula (see e.g. \cite[\S~8.327.1]{GR15})
    \begin{equation}\label{St}
        \Gamma(r+1) \sim \left( \frac{r}{e}\right)^{r+\frac{1}{2}}, \qquad \forall \, r\ge 1, 
    \end{equation}
    as well as $\Gamma(r+1) = r\,\Gamma(r)$ whenever $r>0$ to get that 
    \begin{equation*}
       \frac{\Gamma(u+k_j)}{\Gamma(k_j)}  \sim  \frac{\Gamma(u+k_j+1)}{\Gamma(k_j+1)}\sim e^{-u} \frac{(u+k_j)^{u+k_j+\frac{1}{2}}}{k_j^{k_j+\frac{1}{2}}} = k_j^{u} \, e^{(u+k_j+\frac{1}{2}) \ln{(1+ \frac{u}{k_j})} - u} \sim k_j^u ,
    \end{equation*}
    which holds uniformly in $1\le j\le \ell$ and $0\le u \le 100 \min_{1\le j\le \ell} k_j^{1/2}$.
    Similarly, we have $\frac{\Gamma(n+1)}{\Gamma(u+n+1)} \sim n^{-u}$. This finishes the proof of (i).

    Now we prove (ii). Let us begin with the case where $n\lesssim 1$. Then it follows from \eqref{equ-m} that $m\lesssim 1$, and \eqref{beta} is trivial. 
    
    In the opposite case where $n$ is large enough, so $n > 2(1+m)$. In this case, the second ``$\sim$" is due to \eqref{Gamma} as well as the Stirling's formula \eqref{St}.
    
    Aiming now at the first ``$\sim$" in \eqref{beta}. ``$\le$'' is trivial.  For ``$\gtrsim$'', using the change of variables $s=1-r^2$, we obtain
    \begin{equation*}
        I_{m,n}:= \int_{ ( 1+ \frac{2m}{n} )^{-\frac{1}{2}} }^1 \, (1-r^2)^{\frac{m}{2}} r^{2n+ v -1} \, dr = \frac{1}{2} \int_0^{\frac{2m}{n+2m}} (1-s)^{n+\frac{v}{2}-1} \, s^{\frac{m}{2}}\, ds. 
    \end{equation*}

    By \eqref{equ-m} and Taylor's formula $\log(1-s)= -s+O(s^2)$ at $0$, it holds that
    $$
    (1-s)^{n+\frac{v}{2}-1} = e^{(n+\frac{v}{2}-1) \log(1-s)} \sim e^{-(n+\frac{v}{2}-1)s}, \qquad \forall \, 0 < s < \frac{2m}{n+2m}.
    $$
    Thus,
    \begin{equation*}
        I_{m,n} \sim \int_0^{\frac{2m}{n+2m}}e^{-(n+\frac{v}{2}-1)s} s^{\frac{m}{2}} \,ds = (n+\tfrac{v}{2}+1)^{-\frac{m}{2}-1} \int_0^{m \, \frac{2n+v-2}{n+2m}} e^{-\rr} \, \rr^{\frac{m}{2}} \, d\rr.
    \end{equation*}
    Note that $(n+\frac{v}{2}+1)^{-\frac{m}{2}-1}\sim n^{-\frac{m}{2}-1}$ by \eqref{equ-m}. And $n > 2(1+m)$ gives $\frac{2n+v-2}{n+2m}>1$. 
    In view of the second ``$\sim$" in \eqref{beta}, it remains to show that:
    \begin{equation}\label{incomp}
        \int_0^m s^{\frac{m}{2}} e^{-s} \, ds \sim \Gamma(\tfrac{m}{2}+1), \qquad \forall \, m \ge 1.
    \end{equation}
Indeed it is trivial for $m\lesssim 1$. For $m$ large enough, the standard trick of $e^{-\rr} = e^{-\rr/2} \, e^{-\rr/2}$ implies that 
    \begin{equation*}
        \int_m^\infty s^{\frac{m}{2}} e^{-s} \, ds \le e^{-\frac{m}{2}} \int_m^\infty s^{\frac{m}{2}} e^{-\frac{s}{2}}\, ds = 2 \left(\frac{2}{e} \right)^{\frac{m}{2}}  \int_{\frac{m}{2}}^\infty \rr^{\frac{m}{2}} e^{-\rr} \, d\rr < \frac{1}{2} \int_0^{+\infty} \rr^{\frac{m}{2}} e^{-\rr} \, d\rr.
    \end{equation*}
    Then we also get \eqref{incomp}.  
\end{proof}

\subsection{A useful Lemma}\label{ssec-lemma}

For $R>0$ and $\nu \ge 0$, we set
\begin{equation}\label{equ-unitball}
\mathbf{D}_{\nu} := \fint_{B_1}  \,  \left| \mathbb{W}x \right|^{\nu}   dx = \fint_{B_1} \left( \sumj a_j^2 \, |\xj|^2 \right)^{\frac{\nu}{2}}  dx.
\end{equation}
Uniform estimates for $\mathbf{D}_{\nu}$ w.r.t. $(a_1, \ldots, a_{\ell}, \nu)$ will play a key role in the work. In the special case of classical H-type groups, it holds that $|\mathbb{W}x|=|x|$, and  one easily get $\mathbf{D}_\nu = \frac{2n}{2n+\nu}$ by polar coordinates.  To deal with the general cases, we need to introduce a convex function.
 
\begin{lemma}\label{lem-WW}
 With the dimension restriction $(n,m)$ in \eqref{equ-W}-\eqref{equ-m}, and the constant $\CCC$ defined by \eqref{equ-KCH},  
we have uniformly in $1\le \nu\le 10m$ and $0 < a_1 < \ldots < a_{\ell} = 1$ that
$$
\mathbf{D}_{\nu}  \, \sim \, \CCC^{ \, \frac{\nu}{2}}.
$$ 
\end{lemma}
\begin{proof}
First of all, from the symmetry, we get  
\begin{align}
  \fint_{B_1} x_q^2 \, dx  = \frac{1}{2n} \fint_{B_1}|x|^2 \, dx = \frac{1}{2n+2} \ (q = 1, \ldots, 2 n), \quad \mbox{so} \quad   \mathbf{D}_2 = \frac{n}{n+1}\CCC.
\end{align}

Next we prove $\mathbf{D}_{\nu}   \gtrsim \CCC^{ \, \frac{\nu}{2}}$. If $10 m \ge \nu \ge 2$,   H\"{o}lder's  inequality gives  $\mathbf{D}_{\nu} \ge \mathbf{D}_{2}^{ \, \frac{\nu}{2}} \sim \CCC^{\frac{\nu}{2}}$ in view of the dimensional restriction \eqref{equ-m} and the Stirling's formula.
In the opposite case where $1\le \nu < 2$, reverse Minkowski inequality implies that:
\begin{equation*}
|B_1|\, \mathbf{D}_{\nu}   =   \left\| \sum_{j=1}^{\ell} a_j^2 \, |\xj|^2 \right\|_{L^{\frac{\nu}{2}}(B_1 )}^{\frac{\nu}{2}} \ge \ \left( \sum_{j=1}^{\ell} a_j^2 \, \left\| |\xj|^2 \right\|_{L^{\frac{\nu}{2}}(B_1)} \right)^{\frac{\nu}{2}}. 
\end{equation*}
Then we apply Lemma \ref{lem-a} (i) to get the desired estimate.

Aim now at the upper bound $\mathbf{D}_{\nu} \lesssim \CCC^{ \, \frac{\nu}{2}}$.  It follows from H\"{o}lder's inequality that $\mathbf{D}_{\nu}^2 \le \mathbf{D}_{2\nu}$. Then it suffices to show
\begin{equation*}
\mathbf{D}_{ 2\nu} \mathrm{Tr}\left( \mathbb{W}^2 \right)^{-\nu} =   \fint_{B_1} \, \left( 
\sum_{j=1}^{\ell} \frac{a_j^2 }{\mathrm{Tr} \left( \mathbb{W}^2 \right)} \, |\xj|^2  \right)^{\nu}  dx \, \lesssim \, (2n)^{-\nu}.
\end{equation*}
Consider the following function $\mathfrak{F}$ on $\R_+^\ell = [0,\infty)^\ell$,
\begin{equation*}
w= (w_1, \cdots, w_{\ell}) \, \longmapsto \mathfrak{F}(w) =  \fint_{B_1} \left( 
\sum_{j=1}^\ell w_j |\xj|^2  \right)^{\nu} dx. 
\end{equation*}
It is convex because $\nu \ge 1$. We have
$$
\mathbf{D}_{2 \nu} \mathrm{Tr}\left( \mathbb{W}^2 \right)^{-\nu} \le \max_{w \in \mathscr{K} } \mathfrak{F}(w) \ \ \mbox{with} \ \  \mathscr{K} = \left\{ q \in \R_+^\ell: \ \sumj 2 k_j q_j =1 \right\} .
$$
The maximum is attained at some extreme point of $\mathscr{K}$,  in other words, 
\begin{equation*}
\mathbf{D}_{ 2 \nu} \mathrm{Tr}\left( \mathbb{W}^2 \right)^{-\nu} \le \max_{1\le j \le \ell} \ \frac{1}{ (2k_j)^{\nu }} \fint_{B_1} |\xj|^{2 \nu} dx \, \sim \, (2n)^{-\nu}
\end{equation*}
by means of Lemma \ref{lem-a} (i), which finishes the proof.
\end{proof}

\begin{corollary}\label{corb}
     We have uniformly in $\boldsymbol{\beta}=(\beta_1,\cdots, \beta_m) \in [0,\infty)^m$ and $\frac{1}{2}\le \alpha \le 1$ that
\begin{equation*}
    \int_{B_1}  \, \prod_{l=1}^{m} \left(  |\mathbb{W}x|^2 + \beta_l \right)^{\alpha} \, dx  \sim |B_1| \, \prod_{l=1}^{m} \left( \CCC + \beta_l \right)^{\alpha}.
\end{equation*}
\end{corollary}
\begin{proof}
Let $d_l = d_{l}(\boldsymbol{\beta},m) \ge 0 \, (1\le l \le m)$ be determined by
 \begin{align}\label{BEa1}
\prod_{l=1}^m \left( u^2 + \beta_l \right) = \sum_{l=0}^m d_l \, u^{2l}, \qquad  \forall \, u \in \R.
\end{align}
Combining this with H\"{o}lder inequality, we obtain that:
\begin{align*}
\int_{B_1}  \prod_{l=1}^{m} \left(  |\mathbb{W}x|^2 + \beta_l \right)^{\alpha} \, dx &\le |B_1|^{1 - \alpha} \left( \int_{B_1}  \prod_{l=1}^{m} (|\mathbb{W}x|^2 + \beta_l) \, dx \right)^{\alpha} \\
&= |B_1|^{1 - \alpha} \left( \int_{B_1}  \sum_{l=0}^{m}  d_l \, |\mathbb{W}x|^{2 l} \, dx \right)^{\alpha} \\
&\lesssim |B_1| \, \left( \sum_{l=0}^{m}  d_l \, \CCC^{l} \right)^{\alpha} = |B_1| \, \prod_{l=1}^{m} \left( \CCC + \beta_l \right)^{\alpha},
\end{align*}
where we have used Lemma \ref{lem-WW} in ``$\lesssim$''. 

Similarly, together \eqref{BEa1} with reverse Minkowski inequality, we have that:
\begin{align*}
\int_{B_1}  \prod_{l=1}^{m} \left(  |\mathbb{W}x|^2 + \beta_l \right)^{\alpha} \, dx &\ge \left( \sum_{l = 0}^m d_l \, \| |\mathbb{W}x|^{2 l} \|_{L^{\alpha}(B_1)} \right)^{\alpha} \gtrsim |B_1| \, \prod_{l=1}^{m} \left( \CCC + \beta_l \right)^{\alpha},
\end{align*}
which completes the proof.
\end{proof}

\subsection{The first step of proof of Theorem \ref{thm-voll}}\label{ssec-pfvol}
Inspired by \cite[Lemma 5.3]{BLZ25} and Theorem \ref{thm-Fx}, in $(x,\theta)$-coordinates, the main contribution of the volume comes from the region $|x|\approx R$ and $|\theta| \ll 1$. We will introduce a family of ellipsoidal surfaces which enables us to appropriately partition the region. Moreover,
the condition \eqref{equ-m} and Stirling’s formula will be used repeatedly without mention. 

Recall that $a_\ell =1$ and $K$ is defined by \eqref{equ-KCH}. For $R>0$ and $\delta\in (0,\pi)$, we define 
\begin{equation}
S_{\delta,R} := \left\{  x \in B_R  \, : \, \sumj |\xj|^2 \left( \frac{a_j \delta}{\sin{(a_j \delta})} \right)^2 =R^2  \right\}.
\end{equation}
We have the following observations. First,
 $S_{\delta_1,R} \cap S_{\delta_2,R} = \emptyset$ provided $\delta_1 \ne \delta_2$. Second,   $x \in S_{\delta,R}$ if and only if $x/R \in S_{\delta,1}$. Third, 
\begin{equation*}
B_R \backslash \left\{x \in \R^{2n} : x_{(\ell)}=0 \right\} \subset \bigcup_{0<\delta <\pi} S_{\delta,R}  \subset B_R.
\end{equation*}
Let $\mathcal{B}_{x,R} := \{ t \in \R^m : d_{\mathbb{B}}(x,t) < R \} $,  and then we can write 
\begin{align}
|B_{\mathbb{B}}(o,R)| = \int_{\Sa} \int_{\mathcal{B}_{x,R}} \, dt \, dx + \int_{ \Sb } \int_{\mathcal{B}_{x,R}} \, dt \, dx =: \boldsymbol{\mathbf{O}_1} + \boldsymbol{\mathbf{O}_2},
\end{align}
where we set 
\begin{align}
\Sa := \mathop{\bigcup}\limits_{0<\delta < \sqrt{ \K }} S_{\delta,R}, \qquad \Sb :=   \mathop{\bigcup}\limits_{\sqrt{ \K }< \delta < \pi} S_{\delta,R},
\end{align}
while we can see from the following proof that it is enough to take
\begin{equation}\label{KK}
    \K := \min\{\varepsilon, k_\ell^{-\frac{3}{4}} \},\quad \mbox{with an \ }\varepsilon>0 \ \mbox{sufficiently small.}
\end{equation}

\subsection{Bounds of $\boldsymbol{\mathbf{O}_1}$}\label{sssec-main}

 Our goal is to show that $\boldsymbol{\mathbf{O}_1}$ admits the desired estimates, i.e.  
\begin{align} \label{O1}
   \boldsymbol{\mathbf{O}_1} \sim |B_{\R^{2n+m}}(0, R)| \ \det \left( \frac{R^2 \, \CCC}{12}  \, \mathbb{I}_m + \mathbb{B}^T \mathbb{B} \right)^{\frac{1}{2}}.
\end{align}
To do this, by slightly modifying the arguments in \cite[p. 3788]{BLZ25}, the problem can be reduced to establishing the estimate \eqref{equ-IB} below.
In brief, {\color{blue}due to} the fact that $(\frac{r}{\sin r})^2 = 1 + \frac{r^2}{3} + O(r^4)$ for $r$ near $0$,  there exist constants $0<c_1<c_2$ such that 
\begin{equation}\label{lOl}
	 \int_{ 0 < R^2 -|x|^2 < c_1 \K  \, |\mathbb{W}x|^2  } |\mathcal{B}_{x,R}| \,  dx \le \boldsymbol{\mathbf{O}_1} \le \int_{ 0 < R^2 -|x|^2 < c_2 \K \,  |\mathbb{W}x|^2  } |\mathcal{B}_{x,R}| \,  dx.
\end{equation}

For $ R^2 -|x|^2 < c_2 \K \,  |\mathbb{W}x|^2$, we estimate the volume of $\mathcal{B}_{x,R}$ by means of the diffeomorphism $\Fx$ defined by \eqref{equ-Atheta}. The key point is that, in our setting, it follows from \cite[Theorem~VI.7.1]{Bhatia97} that the Jacobian determinant of $\Fx$ satisfies 
\begin{align}
    \det \big( D \Fx(\theta) \big) \sim \det \big( D \Fx(0) \big) = 6^{-m} \, \prod_{j=1}^m \left( \left| \mathbb{W}x \right|^2 + 12 b_j\right).
\end{align}
And a direct calculation yields that
\begin{align}
	  | \mathcal{B}_{x,R} |  \sim   C_m \, (R^2 -|x|^2)^{\frac{m}{2}}  \prod_{j=1}^m \left( \left| \mathbb{W}x \right|^2 + 12 b_j\right)^{\frac{1}{2}} \ \ \mbox{with} \ \ C_m = \left( \frac{1}{12} \right)^{\frac{m}{2}} \frac{|\mathbb{S}^{m-1}|}{m}.
\end{align}
By \eqref{lOl}, a scaling gives
$
\boldsymbol{\mathbf{O}_1} \sim C_m \, R^{2n+2m} \ \mathbf{I}_\mathbb{B},
$ 
with (for simplicity, we can omit the constant $c_i$)
\begin{equation}
\mathbf{I}_\mathbb{B} = \int_{  1 -|x|^2 < \K \, |\mathbb{W}x|^2  } (1 -|x|^2)^{\frac{m}{2}} \prod_{j=1}^m \left( |\mathbb{W}x|^2 + \frac{12 b_j}{R^2} \right)^{\frac{1}{2}} \, dx.
\end{equation}
Therefore, to prove \eqref{O1}, it remains to show
\begin{equation}\label{equ-IB}
\mathbf{I}_\mathbb{B} \sim \frac{\Gamma(\frac{m}{2}+1)}{n^{\frac{m}{2}+1}} \, |\mathbb{S}^{2n-1}| 
\prod_{j=0}^{m} \left( \CCC + \frac{12 b_j}{R^2} \right)^{\frac{1}{2}}    
=: E_{m,n}  \prod_{j=0}^{m} \left( \CCC + \frac{12 b_j}{R^2} \right)^{\frac{1}{2}},
\end{equation}
for which we need a lemma.
\begin{lemma}\label{lem-I0j}
   Let $\varrho$ be as in \eqref{KK}. We have, uniformly in $j\in \{0,1,\cdots,m\}$, that
   \begin{equation}\label{equ-jj}
\mathbf{M}_{0,j} := 
\int_{ 1 -|x|^2 < \K \, |\mathbb{W}x|^2  } (1-|x|^2)^{\frac{m}{2}} \, |\mathbb{W}x|^{j} \, dx \sim E_{m,n} \, \CCC^{\frac{j}{2}},
\end{equation}
where the implicit constant may depend on $\varepsilon$ in \eqref{KK}.
\end{lemma}

\begin{proof}
Let $d\sigma$ be the standard surface measure on $\mathbb{S}^{2n-1}$. Using polar coordinates, we obtain
\begin{align}\label{equ-polar-coor}
\begin{split}
 &\mathbf{D}_j  = \frac{1}{(2n+j) \, |B_1|} \int_{\mathbb{S}^{2n-1}} |\mathbb{W} w|^j \, d\sigma(w), \quad \mbox{and} \\
 &\mathbf{M}_{0,j}  = \int_{\mathbb{S}^{2n-1}} |\mathbb{W} w|^j  \left(  \int_{ (1+ \K \, |\mathbb{W}w|^2 )^{-\frac{1}{2}} }^1 \, (1-r^2)^{\frac{m}{2}} r^{2n+j-1} \, dr \right) d\sigma(w).
\end{split}
\end{align}

Thus, Lemma \ref{lem-WW} gives 
\begin{equation}\label{equ-Sj}
    \int_{\mathbb{S}^{2n-1}} |\mathbb{W} w|^j \, d\sigma(w)  \sim \CCC^{\frac{j}{2}}\, (2n+j)\, |B_1| \sim \CCC^{\frac{j}{2}} \, |\mathbb{S}^{2n-1}|.
\end{equation}
 Combining this with \eqref{beta}, we can easily get the desired upper bound in \eqref{equ-jj}.

For the lower bound, notice that 
\begin{align}\label{equ-varE}
\begin{split}
    \mathbf{M}_{0,j} \ge & \int_{\mathcal{E}} |\mathbb{W} w|^j  \left(  \int_{ ( 1+ \kappa_0 )^{-\frac{1}{2}} }^1 \, (1-r^2)^{\frac{m}{2}} r^{2n+j-1} \, dr \right)  d\sigma(w), \\
    &\mbox{where} \quad  \mathcal{E}=\{ w\in \mathbb{S}^{2n-1}: \K \, |\mathbb{W}w|^2 \ge \kappa_0 \},
\end{split}
\end{align}
for some parameter $\kappa_0>0$ to be chosen momentarily. We consider two cases $k_{\ell}\gg 1$ and $k_{\ell} \lesssim 1$. 

\textbf{Case I. $k_{\ell}\gg 1$. } In such case, we have $\varrho = k_{\ell}^{-\frac{3}{4}} \ll 1$. Choosing $\kappa_0 = 2 \, m \, n^{-1}$ such that $\mathcal{E} \ne \emptyset$.

Let us begin with the case where $j \in \{1, \ldots, m\}$. It follows from \eqref{beta} that:
\begin{align}\label{MM}
\mathbf{M}_{0,j} \gtrsim  \frac{\Gamma(\frac{m}{2}+1)}{n^{\frac{m}{2}+1}} \int_\mathcal{E}  \, |\mathbb{W}w|^j \, d\sigma(w) \sim \frac{\Gamma(\frac{m}{2}+1)}{n^{\frac{m}{2}+1}} \int_{\mathbb{S}^{2n-1}}  \, |\mathbb{W}w|^j \, d\sigma(w) \sim E_{m,n} \, \CCC^{\frac{j}{2}},
\end{align}
where the first ``$\sim$" is due to  the following inequality
\begin{equation}\label{equ-EE}
    \int_{ \mathbb{S}^{2n-1} \backslash \mathcal{E}}  \, |\mathbb{W}w|^j \, d\sigma(w) \le \left(  \K^{-1} \kappa_0 \right)^{\frac{j}{2}}|\mathbb{S}^{2n-1} \backslash \mathcal{E}| \ll \int_{\mathbb{S}^{2n-1}}  \, |\mathbb{W}w|^j \, d\sigma(w),
\end{equation}
while ``$\ll$" in \eqref{equ-EE} 
comes from the fact that $ \K^{-1} \kappa_0 \ll \CCC$ by \eqref{equ-m}.  

For the opposite case $j = 0$.
It suffices to show
\begin{equation*}
II :=  \int_{ 1 -|x|^2 \, \ge \, \K \, |\mathbb{W}x|^2  } (1-|x|^2)^{\frac{m}{2}} \, dx \ll E_{m,n}.
\end{equation*}
In fact, H\"{o}lder's inequality implies that: 
\[
II \le \left( \int_{B_1} (1-|x|^2)^{m} \,dx  \right)^{\frac{1}{2}} \left( \int_{1 -|x|^2 \, \ge \, \K \, |\mathbb{W}x|^2} dx \right)^{\frac{1}{2}}.
\]
By polar coordinates and applying the second ``$\sim$ " in \eqref{beta} with $2 m$ instead of $m$, the first term in the right hand side has order $ E_{2m,n}^{\frac{1}{2}} \sim \left(\frac{\Gamma(m+1)}{n^m} \, |B_1| \right)^{\frac{1}{2}}$. For the second one, a simple calculation shows that it equals 
\begin{align} \label{Est1}
\exp\left\{-\frac{1}{2} \sum_{j = 1}^{\ell} k_j \log(1 + a_j^2 \, \K) \right\} \, | B_1 |^{\frac{1}{2}} \le e^{ -\frac{1}{4} \K\, k_\ell} \, |B_1|^{\frac{1}{2}},
\end{align}
which implies immediately the desired lower bound for $k_{\ell} \gg 1$.

\textbf{Case II. $k_{\ell}\lesssim 1$.} In such case, we have $\varrho = \varepsilon \ll 1$ and take $\kappa_0 = \varepsilon^2 n^{-1}$.
Since $m\lesssim k_\ell \lesssim 1$, estimates of type \eqref{beta} (with $2mn^{-1}$ replaced by $\varepsilon^2 n^{-1}$) still holds, with implicit constants depending on $\varepsilon$.  The proof in \textbf{Case I} involving \eqref{MM}-\eqref{equ-EE} can be applied here. Indeed, if $j\ge 1$, we use the fact $\varrho^{-1}\kappa_0 \ll \CCC$ to get \eqref{equ-EE}; if $j=0$, the proof can be simplified, for which we obtain \eqref{equ-EE} via polar coordinates.
\end{proof}

Now we are in a position to prove \eqref{equ-IB}.
Note that if $\mathbb{B} = 0$, we have $\mathbf{I}_{\mathbb{B}}=\mathbf{M}_{0,m}$ and Lemma \ref{lem-I0j} applies.
Thus, we are left to prove the case $\mathbb{B} \neq 0$.

Let us begin with the lower bound. Similar to \eqref{BEa1}. Let $d_j := d_j(R, \A) \ge 0$ ($0 \le j \le m$) denote the coefficients determined by
\begin{align}  \label{BEa}
\prod_{j=1}^m \left( u^2 + \frac{12 b_j}{R^2} \right) = \sum_{j=0}^m d_j \, u^{2j}, \qquad  u \in \R.
\end{align}

By the reverse Minkowski inequality, we get that
\begin{equation}\label{equ-Min}
\mathbf{I}_\mathbb{B}^{\, 2} \, \ge \,  \sum_{j=0}^m d_j \left( \int_{ 1 -|x|^2 < \K \, |\mathbb{W}x|^2  } (1-|x|^2)^{\frac{m}{2}} \, |\mathbb{W}x|^{j} \, dx \right)^2,
\end{equation}
 which, together with Lemma \ref{lem-I0j}, implies that
\[
\mathbf{I}_\mathbb{B} \gtrsim  E_{m,n} \left( \sum_{j=1}^m d_j \, \CCC^j \right)^{\frac{1}{2}}  = E_{m,n} \prod_{j=0}^{m} \left( \CCC + \frac{12 b_j}{R^2} \right)^{\frac{1}{2}}.
\]

We turn to the upper bound of $\mathbf{I}_\mathbb{B}$. In fact, using  Lemma \ref{lem-I0j} with $j=0$ and H\"{o}lder's inequality, we deduce
\begin{align*}
\mathbf{I}_\mathbb{B}^{\, 2} &\lesssim E_{m,n}  \int_{ 1 -|x|^2 < \K \, |\mathbb{W}x|^2  } (1-|x|^2)^{\frac{m}{2}} \prod_{j=1}^m \left( |\mathbb{W}x|^2 + \frac{12 b_j}{R^2} \right) \, dx \\
&\lesssim  E_{m,n}^2 \, \prod_{j=1}^m \left( \CCC + \frac{12 b_j}{R^2} \right),
\end{align*}
by repeating the above argument using \eqref{BEa}. These estimates imply \eqref{equ-IB}, which concludes the proof of \eqref{O1}.

\subsection{Estimations for $\boldsymbol{\mathbf{O}_2}$}\label{I_2}

Let $b_0 =0$. For any $0\le l \le m$ and $x\in \mathop{\cup}\limits_{0<\delta<\pi} S_{\delta,R}$. There exists a unique $\widetilde{\theta}_l = \widetilde{\theta}_l (x) \in (0,\pi)$ such that
\begin{equation}\label{thetaL}
 \sumj |\xj|^2 \left( \frac{a_j \widetilde{\theta}_l }{\sin{(a_j \widetilde{\theta}_l })} \right)^2 + 4 b_l \, \widetilde{\theta}_l^{\, 2} =R^2.
\end{equation}
Using the inequality $\frac{r}{\sin r} \ge 1+ \frac{r^2}{6}$ for $0 < r < \pi$ (see \cite[eq. (5.7)]{BLZ25}),  we have
\[
\widetilde{\theta}_l \le \frac{1}{2} (R^2 - |x|^2 )^{\frac{1}{2}} \left( 
\frac{1}{12} |\mathbb{W}x|^2 + b_l  \right)^{-\frac{1}{2}}, \qquad 0 \le l \le m.
\]
 Let $g=(x,t)$ with $x\in \mathop{\cup}\limits_{0<\delta<\pi} S_{\delta,R}$ satisfying $d_\mathbb{B}(g)<R$ and $x_{(\ell)}\ne 0$, and $\theta = \theta(g)$ be as in \eqref{equ-Atheta}.
Then by the expression of $d_\mathbb{B}$ (see \eqref{equ-dBB}) and the monotonicity of the function $ \frac{v}{\sin v}$ on $(0,\pi)$, we know $|\theta| \le \widetilde{\theta}_0$ and $|\theta_l| \le \widetilde{\theta}_l$ ($1 \le l \le m$).
 Since $\widetilde{\mu}$ is increasing on $(0, \  \pi)$ (cf. \eqref{MU}), we obtain that:
 \begin{align*}
 	|\mathcal{B}_{x,R}| \le \prod_{l=1}^m \widetilde{\theta}_l \left( 
 	\sumj a_j^2 \, |\xj|^2 \, \widetilde{\mu}(a_j \widetilde{\theta}_0) + 2 b_l \right) \le c^m \, (R^2 - |x|^2 )^{\frac{m}{2}} \, \mathcal{K}(x;\widetilde{\theta}_0 (x)),
 \end{align*}
where
\begin{equation}\label{mathK}
\mathcal{K}(x;r) = \prod_{l=1}^m \left( \sumj a_j^2 \, |\xj|^2 \, \widetilde{\mu}(a_j r) +  b_l \right)  \left( |\mathbb{W}x|^2  +  b_l \right)^{-\frac{1}{2} }, \qquad r > 0.    
\end{equation}

Plugging this estimate into $\boldsymbol{\mathbf{O}_2}$ and using the scaling $x \mapsto Rx$, we deduce that
\begin{align*}
\boldsymbol{\mathbf{O}_2} 
& \le c^m R^{2n+m}   \int_{\SBb} (1-|x|^2)^{\frac{m}{2}} \, \mathcal{K}\left( Rx ; \widetilde{\theta}_0 (Rx) \right) \, dx  =: c^m R^{2n+m} \mathbf{K}.
\end{align*}

Next, in view of \eqref{equ-m} and by considering separately the cases where $m \lesssim 1$ and $m\gg 1$, we can choose $\gamma \in (1,2)$ such that 
\begin{equation}\label{equ-mr}
    2 k_\ell > \gamma \, m + \frac{1}{2}.
\end{equation} 
For instance, we can take $\gamma = \frac{201}{200}$.
Then H\"{o}lder's inequality gives 
\begin{equation*}
    \mathbf{K} \le   \left( \int_{B_1} (1-|x|^2)^{\frac{m\gamma'}{2}} \, dx \right)^{1/\gamma'}  \left(  \int_{\SBb} \mathcal{K} \big(Rx ; \widetilde{\theta}_0 (Rx) \big)^{\gamma} \, dx \right)^{1/\gamma} =: {\mathbf{K}}_1 \, {\mathbf{K}}_2.
\end{equation*}

First, by polar coordinates, we get
\begin{equation}\label{equ-K1}
    {\mathbf{K}}_1 \lesssim c^m \, \Gamma(\tfrac{m}{2}+1) \, n^{-\frac{m}{2}} |B_1|^{1/\gamma'}.
\end{equation}
To estimate ${\mathbf{K}}_2$, we decompose the domain into annuli. More precisely, 
let $\delta_0 = \K^{\frac{1}{2}}$ and $\delta_k = \pi - 2^{-k}$ ($k\ge 1$), and then
\[
{\mathbf{K}}_2^\gamma = \sum_{k=0}^{\infty} \int_{ \bigcup_{  \delta_k < \delta < \delta_{k+1} } S_{\delta,1} }  \mathcal{K} \big(Rx;\widetilde{\theta}_0 (Rx) \big)^{\gamma} \, dx =: \sum_{k=0}^{\infty} J_k.  
\]
Since $\mathcal{K}(x;\cdot)>0$ is increasing on $(0,\pi)$ and $\widetilde{\theta}_0(Rx) = \delta $ for $x\in S_{\delta,1}$, it follows that
\begin{align*}
J_k \le \int_{ \bigcup_{  \delta > \delta_k   } S_{\delta,1} } \mathcal{K}(Rx;\delta_{k+1})^\gamma \, dx = \int_{ \sumj |\xj|^2 \left(  
\frac{a_j \delta_k}{\sin{(a_j \delta_k})} \right)^2 < 1 } \mathcal{K}(Rx; \delta_{k+1})^\gamma \, dx.
\end{align*}
From \eqref{MU} and the elementary equality 
\begin{align} \label{EIa}
\frac{\sin{r}}{r} = \prod_{j = 1}^{+\infty} \left(1 - \frac{r^2}{j^2 \pi^2} \right), \qquad r \in \R,
\end{align}
we {have uniformly in $j,k\ge 1$ that}:
\begin{equation}\label{equ-sind}
    \frac{\sin{(a_j \delta_k)}}{a_j \delta_k} \ge \frac{\sin  \delta_k}{ \delta_k} \quad \mbox{and} \quad  \left( \frac{\sin{(a_j \delta_k)}}{a_j \delta_k} \right)^2 \widetilde{\mu}(a_j \delta_{k+1}) \lesssim 1.
\end{equation}
Then
\begin{align*}
   \prod_{l=1}^m (R^2 \, |\mathbb{W}x|^2 + b_l)^{-\frac{1}{2}} &= \prod_{l=1}^m \left(  \frac{\sin \delta_k}{\delta_k} \right)^{-1} \left[ R^2 \, \sumj a_j^2 \, |\xj|^2 \,\left(  
\frac{\delta_k}{\sin{\delta_k}} \right)^2 +  b_l \left(  
\frac{\delta_k}{\sin{\delta_k}} \right)^2 \right]^{-\frac{1}{2}} \\
&\le \left(  \frac{\sin \delta_k}{\delta_k} \right)^{-m} \prod_{l=1}^m \left(  
R^2 \, \sumj a_j^2 \, |\xj|^2 \,\left( \frac{a_j \delta_k}{\sin{(a_j \delta_k)}} \right)^2 +  b_l \right)^{-\frac{1}{2}},
\end{align*}
and 
\begin{align*}
    \prod_{l=1}^m \left[ R^2 \, \sumj a_j^2 \, |\xj|^2 \, \widetilde{\mu}(a_j \delta_{k+1}) +  b_l  \right] \le C^m \, \prod_{l=1}^m \left[ R^2 \, \sumj a_j^2 \, |\xj|^2 \,\left( \frac{a_j \delta_k}{\sin{(a_j \delta_k)}} \right)^2 +  b_l \right].
\end{align*}
Hence we obtain
\begin{align*}
  \mathcal{K}(Rx; \delta_{k+1}) \le C^m  \, \left( \frac{\sin \delta_k}{\delta_k} \right)^{-m  } \, \prod_{l=1}^m \left(  
R^2 \, \sumj a_j^2 \, |\xj|^2 \,\left(  
\frac{a_j \delta_k}{\sin{(a_j \delta_k)}} \right)^2 +  b_l \right)^{\frac{1}{2}}.
\end{align*}
Now perform linear transform $\left\{ y_{(j)} =  \frac{a_j \delta_k}{\sin{(a_j \delta_k)}} \xj \right\}_{j=1}^\ell $. This yields 
\[
J_k \le c^m \left( \frac{\sin \delta_k}{\delta_k} \right)^{\frac{1}{2}} \, \left( \frac{\sin \delta_k}{\delta_k} \right)^{2 k_{\ell} - m\gamma - \frac{1}{2}} \int_{B_1}  \, \prod_{l=1}^{m} \left( R^2 \, |\mathbb{W}y|^2 + b_l \right)^{\gamma/2} \, dy.
\]
 By Corollary \ref{corb}, we have 
\[
\int_{B_1}  \, \prod_{l=1}^{m} \left( R^2 \, |\mathbb{W}y|^2 + b_l \right)^{\gamma/2} \, dy  \sim |B_1| \, \prod_{l=1}^{m} \left(R^2 \, \CCC + b_l \right)^{\gamma/2}.
\]
Moreover, by \eqref{equ-mr} and the elementary inequality $\frac{\sin r}{r} \le e^{-\frac{1}{120} r^2 }$ whenever $r\in (0,\pi)$, we know
\begin{align*}
J_k &\le C^m e^{-\frac{1}{120} (2k_\ell -m \gamma -\frac{1}{2}) \K} |B_1| \left(\frac{\sin \delta_k}{\delta_k}\right)^{\frac{1}{2}} \prod_{L=1}^{m} \left(R^2 \, \CCC + b_L \right)^{\gamma/2}.
\end{align*}
Since $\frac{\sin \delta_k}{ \delta_k } \sim 2^{-k}$, summing in $k$ gives
\begin{equation}\label{equ-K2}
    \mathbf{K}_2^\gamma \le C^m e^{-\frac{1}{120} (2k_\ell -m \gamma -\frac{1}{2}) \K} |B_1| \prod_{L=1}^{m} \left(R^2 \, \CCC + b_L \right)^{\gamma/2}.
\end{equation}
Finally, for any constant $c_0>0$,
by considering separately the cases where $m \lesssim 1$ and $m \gg 1$, one gets that  there exists a constant $D(c_0)$ only depending on $c_0$ (in particular, independent of $m$) such that $c_0^m e^{-\frac{1}{120} (2k_\ell -m \gamma -\frac{1}{2}) \K} \lesssim  D(c_0).$
Combining \eqref{equ-K1} and \eqref{equ-K2}, we conclude $\boldsymbol{\mathbf{O}_2} 
\lesssim \boldsymbol{\mathbf{O}_1} $. The proof of Theorem \ref{thm-voll} is then finished.\\

\section{ Maximal Function on $\HH$ }\label{sec-app}

As in \cite{Li09}, it is enough to establish the following uniform estimates:
\begin{equation}\label{equ-HDS} 
\big|B_{\mathbb{B}}( o , r)\big|^{-1}\chi_{B_{\mathbb{B}}(o,r)}(g)\lesssim n\left( \frac{3}{2 \,\CCC } \right)^{\frac{m}{2}} \fint_{0}^{ \frac{  r }{\sqrt{n}}}   {P}_{h} (g) \, d h, \quad a.e. \ g, \  \forall \, r>0,
\end{equation}
where $P_h \ (h>0)$ denotes the convolution kernel of $e^{-h \sqrt{-\Delta_{\mathbb{B}}}}$.  With the uniform volume estimates in hand, it suffices to establish uniform asymptotic behaviour of $P_h$, see Proposition \ref{prop-poi} below. To do this, we adapt the method used in \cite{BLZ25}.

\subsection{Poisson Kernel}\label{ssec-ker}

Set
\begin{align}\label{equ-cj}
\ N:=n+m+\frac{1}{2}, \quad \mathfrak{c}_j := \frac{k_j}{N}, \quad  j = 1,\cdots, \ell-1, \quad \mathfrak{c}_{\ell} := 1-\sum_{j=1}^{\ell-1} \mathfrak{c}_j.
\end{align}

Using the subordinate formula, the Poisson kernel $P_h \ (h>0)$  is given by: 
\begin{equation}\label{equ-poisson}
P_h(g) = \frac{2^m \Gamma(N)}{\pi^N} h \int_{\R^m} Q_h(g;\lambda)\,d\lambda,
\end{equation}
where 
\begin{align*}
Q_h(g;\lambda) &= \mathbf{V}(\lambda) \left(\Phi(g;\lambda)+h^2\right)^{-N}= \left(\frac{\sinh|\lambda|}{|\lambda|}\right)^{m+\frac{1}{2}} S_h(g;\lambda)^{-N}, \\
S_h(g;\lambda) &= \left(\Phi(g;\lambda)+h^2\right) \prod_{j=1}^{\ell} 
\left(\frac{\sinh(a_j |\lambda|)}{a_j |\lambda|}\right)^{c_j}, \quad  h > 0.
\end{align*}

We first introduce some notation which looks a little complicated due to the lack of symmetry.
For $0 \le r < \pi$, set in the sequel (cf. \eqref{EIa})
\begin{gather}\label{equ-f'}
\Lambda (r):= \prod_{j=1}^{\ell} \left(\frac{\sin(a_j r)}{a_j r}\right)^{\mathfrak{c}_j}, \quad  \varphi(r) := \frac{d}{d r} \ln{\Lambda (r)} = -r \sumj \sum_{k=1}^{+\infty} \frac{2a_j^2 \mathfrak{c}_j}{ (k \pi)^2 - (a_j r)^2}. \\
\CG(r) := 1 + r \, \varphi(r) = \sumj \mathfrak{c}_j \, (a_j r) \, \cot{(a_j r)}, \quad u(r) := - \frac{\varphi(r)}{r \, \CG(r)}. \label{fCG}
\end{gather}
Moreover, for given $x \in \R^{2n}$ with $x_{(\ell)} \neq 0$ and $h > 0$, set for $0 \le r < \pi$ and $\lambda \in B_{\R^m}(0, \pi)$:
\begin{align}\label{fTA}
\begin{split}
     & \Tg(r)  := \sumj \left( 
	u(r) \left( \frac{a_j r}{\sin{(a_j r)}} \right)^2 + \frac{a_j \mu(a_j r)}{r}  \right)  \, |\xj|^2 + u(r) \, h^2,  \\
   & \TA (\lambda) := \Tg(|\lambda|) + 4 u( |\lambda| ) \, \lambda^T \mathbb{A} \lambda, \quad \TB (\lambda):= \frac{\Tg'(|\lambda|) + 4  u'(|\lambda|) \, \lal }{|\lambda|}.
\end{split}
\end{align}

Inspired by Theorem \ref{thm-Fx}, let us begin with the maximizer of the following crucial function defined on $B_{\R^m} (0,\pi) $ for given $g$ and $h > 0$:
\begin{align}\label{equ-sS}
\begin{split}
\sS(\lambda) &\equiv \sS_\mathbb{A}(g,h;\lambda) := S_h(g; \mathrm{i} \lambda)  = \Lambda(|\lambda|) \left( \phi(g;\lambda)+h^2 \right), 
\end{split}
\end{align}

Notice that $\sS$ has the form $f_1(|\lambda|)(f_2(\lambda) + 4 t \cdot \lambda)$,  the following simple observation will be used:

\begin{lemma}\label{lem-nabla}
Let $f_1, f_2 \in C^\infty(B_{\R^m}(0,\pi))$ and $F_t( \lambda) := f_1(\lambda) \left( 
f_2(\lambda) + 4 t\cdot \lambda \right).$  If $\lambda_0$ is a critical point of $F_t$ such that $f_1(\lambda_0) \ne 0$ and $f_1(\lambda_0) + \lambda_0 \cdot \nabla f_1 (\lambda_0) \ne 0$, then
\begin{equation}\label{equ-solvet}
4 t = - \nabla f_2(\lambda_0) - \nabla f_1(\lambda_0) \,  \frac{f_2(\lambda_0) - \lambda_0 \cdot \nabla f_2(\lambda_0)}{ f_1(\lambda_0) + \lambda_0 \cdot \nabla f_1(\lambda_0)}.
\end{equation}
Consequently, 
\begin{equation}\label{equ-frakS}
F_t(\lambda_0) = f_1(\lambda_0)^2 \, \frac{f_2(\lambda_0) - \lambda_0 \cdot \nabla f_2(\lambda_0)}{ f_1(\lambda_0) + \lambda_0 \cdot \nabla f_1(\lambda_0)}.
\end{equation}
\end{lemma}

\begin{proof}
The equation $\nabla F_t (\lambda_0) =0$ gives
\begin{equation}\label{equ-cri}
\nabla f_1(\lambda_0) \left( f_2(\lambda_0) + 4 t \cdot \lambda_0 \right) + f_1(\lambda_0) \left( \nabla f_2(\lambda_0) + 4 t \right) =0.
\end{equation}
 By taking the inner product  with $\lambda_0$ on both sides of the last equality, we yield
\begin{align} \label{SE1}
  \Big( f_1(\lambda_0) + \lambda_0 \cdot \nabla f_1 (\lambda_0) \Big) \, 4 t \cdot \lambda_0 + \Big( \lambda_0 \cdot \nabla f_1(\lambda_0) \, f_2(\lambda_0) +  \lambda_0 \cdot \nabla f_2(\lambda_0) \, f_1(\lambda_0) \Big) = 0, 
\end{align}
which allows us to determine $t \cdot \lambda_0$. Substituting it in \eqref{equ-cri}, we get the desired results.
\end{proof}

\subsection{The Maximizer of $\sS$}\label{sec-max}
 The following claim can be easily proved in a more general setting that $(\mathbb{G}, \Delta)$ admit GM-property, by means of the operator convexity. However, we will provide an elementary proof, where we determine the maximizer by means of a diffeomorphism $\MC$ (see \eqref{equ-MC} below). This map is also useful in estimating the Poisson kernel; see Section \ref{sec-hess}.

\begin{prop}\label{prop-max}
The function $\sS$ possesses a unique maximizer in $B_{\R^m}(0,\pi)$. 
\end{prop}

\begin{lemma} \label{Lm43}
Let $x \in \R^{2n}$ with $x_{(\ell)} \neq 0$ and $h > 0$. We have

(1) 
The unique zero $r_*$ of $\CG$ in $(0, \  \pi)$ satisfies $\frac{\pi}{2} \le r_* < \pi$. 
Moreover, the smooth functions $u$ and $\Tg$ are positive and strictly increasing on $[0, \ r_*)$.
 
(2) The following smooth map is a $C^{\infty}$-diffeomorphism from $B_{\R^{m}}(0, r_*)$ onto $\R^m$,
    \begin{align}\label{equ-MC}
     \MC(\lambda) := 4^{-1} \, \Big( \TA(\lambda) \, \mathbb{I}_m  + 8 \mathbb{A}   \Big) \lambda.
    \end{align} 
\end{lemma}

\begin{proof}
    The claim $\frac{\pi}{2} \le r_* < \pi$ is clear by the fact that $\CG$ is strictly decreasing on $(0, \pi)$ (cf. \eqref{equ-f'} and \eqref{fCG}) together with $\CG(\tfrac{\pi}{2}) \ge 0$ and $\CG(\pi^-) = - \infty$. Moreover, \eqref{equ-f'} and \eqref{fCG} also imply that $u$ is positive and strictly increasing on $[0, \ r_*)$. The corresponding assertion for $\Tg$ follows from \eqref{EIa} and \eqref{MU}.

    Aim now at the second assertion. 
    A direct calculation implies that the Jacobian matrix of $\MC$ equals $\frac{1}{4} \mathbb{J}(\lambda)$, where
\begin{align}\label{equ-JJ}
\mathbb{J}(\lambda):= \mathbb{D}_1(\lambda)  + \lambda \lambda^T \, \mathbb{D}_2 (\lambda),  
\end{align}
with
\begin{equation}
    \mathbb{D}_1 (\lambda) = \TA(\lambda) \, \mathbb{I}_m + 8\mathbb{A},\ \ \mathbb{D}_2 (\lambda) = \TB(\lambda) \, \mathbb{I}_m + 8 u(|\lambda|) \mathbb{A}.
\end{equation}
Notice that $\mathbb{D}_1$ and $\mathbb{D}_2$ are real, diagonal matrices. Moreover $\mathbb{D}_2 (\lambda)$ is positive semi-definite, and the positive definite matrix  $\mathbb{D}_1 (\lambda) \ge \Tg(0) \, \mathbb{I}_m > 0$. 
From Schur's Lemma, we get that 
$$\det \mathbb{J}(\lambda) = \big( 1 + \lambda^T \mathbb{D}_2 (\lambda)  \mathbb{D}_1 (\lambda)^{-1} \lambda \big) \det  \mathbb{D}_1 (\lambda) > 0.$$
By applying Hadamard's theorem (cf. e.g. \cite[Theorem 6.2.8]{K02}), it remains to show that $\MC$ is proper, which is clear by the fact that $  \Tg(r) \rightarrow  +\infty$ as $r \rightarrow r_*^-$.
\end{proof}

\begin{proof}[Proof of Proposition \ref{prop-max}]
The existence of the maximizer is due to the facts that
$$
\sS(0)=|x|^2+ h^2 > 0 \quad \mbox{and} \quad \lim_{\ |\lambda|\rightarrow \pi^{-}} \sS(\lambda) = - \infty.$$ Moreover,
$
\mathop{\mbox{lim}}\limits_{\ r\rightarrow 0^{+}} \frac{d}{dr} \sS( r \hat{t} ) = 4|t|   >0,
$
so $0$ cannot be a maximizer. 

 Let $\tau$ be a maximizer, and $\rho := |\tau| \in (0, \pi)$. We first claim that $1 + \rho \frac{\Lambda'(\rho)}{\Lambda(\rho)} \neq 0$. Suppose the contrary.  
  Applying \eqref{SE1} with 
 \[
 f_1(\lambda) = \Lambda(\lambda), \quad f_2 (\lambda) = \phi(g; \lambda) - 4 t \cdot \lambda+h^2,\quad \lambda_0 = \tau,
 \]
it follows from the elementary equality $r \cot{r} - r\frac{d}{d r} (r \cot{r}) = (\frac{r}{\sin{r}})^2$ that 
\begin{equation*}
\sumj |\xj|^2 \left( \frac{a_j \rho}{\sin{(a_j \rho)}} \right)^2 + 4 \tau^T \mathbb{A} \tau + h^2  =0,
\end{equation*}
which leads to a contradiction.

 Moreover it follows from \eqref{equ-frakS} that
\begin{align}\label{equ-Stau}
	( 0 < ) \, \sS(\tau ) 
	& = \frac{\Lambda(\rho)}{1+ \rho \, \varphi(\rho)} \left( h^2 + 4\tau^T \mathbb{A} \tau + \sum_{j = 1}^{\ell} 
	\left( \frac{a_j \rho}{\sin {(a_j \rho)}} \right)^2 \, |x_{(j)}|^2
	\right).
\end{align}
Hence $1 + \rho \, \varphi(\rho) >0$. Thus by \eqref{equ-solvet}, a direct computation shows that $t = \MC(\tau)$. The desired conclusion then follows from Lemma \ref{Lm43}.
\end{proof}

Let $h > 0$. In what follows, we always assume that 
\begin{align}
    g \in \GG := \left\{  (x,t)\in \HH : x_\iota,t_l\ne 0, \ \forall \, \iota, l \right\}.
\end{align} 
Let $\tau_h = \tau_h (g)$ denote the unique maximizer of $\sS$ in $B_{\R^m}(0,\pi)$. Then we have
\begin{gather}
    t = \MC(\tau_h) = 4^{-1} \, \Big( \TA(|\tau_h|) \, \mathbb{I}_m  + 8 \mathbb{A}   \Big) \tau_h, \quad 0 < \rho_h : = |\tau_h| < r_*, \label{equ-ttau} \\
   \sS(\tau_h ) = \frac{\Lambda(\rho_h)}{1 + \rho_h \, \varphi(\rho_h)} \left( h^2 + 4\tau_h^T \mathbb{A} \tau_h +  \sumj
	\left( \frac{a_j \rho_h}{\sin {(a_j \rho_h)}} \right)^2 \, |x_{(j)}|^2
	\right). \label{NST}
\end{gather} 
We further provide some elementary properties of $\tau_h$.

\begin{lemma}\label{lem-he}
The following statements hold uniformly:
\begin{equation*}
\ \ \    (1) \ \phi(g;r \hat{\tau_h}) \ \mbox{is increasing in} \ r\in [0,\rho_h]; \ \ \ \, \; \, (2) \ 	 d_{\mathbb{B}}(g)^2 \lesssim \sS(\tau_h ) \le \left(d_{\mathbb{B}}(g)^2 + h^2 \right) \Lambda(\rho_h); 
\end{equation*}
\begin{equation*}
(3) \, (\pi - \rho_h) \sim 1 \mbox{\, if Assumption \ref{assumption} holds;} \quad \ (4) \ t_l^2 \lesssim (b_l + \sS(\tau_h))\sS(\tau_h), \ 1\le l \le m.\,
\end{equation*}
\end{lemma}

\begin{proof}
We first prove (1).
Notice that
\[
\phi(g; r  \hat{\tau_h} \,)   =   \sumj (a_j r) \cot{(a_j r)} \, |\xj|^2 +  \frac{4 r }{\rho_h} t\cdot \tau_h -  \frac{4r^2}{\rho_h^2} \tau_h^T \mathbb{A}\tau_h.
\]
 From the definition of $\mu$ in \eqref{MU}, it holds that $\frac{d^2}{dr^2} \phi(g; r \tau_h) \le 0$. Hence we only need to show  $ \frac{d}{dr}  \phi(g;r\hat{\tau_h})|_{r=\rho_h} \ge 0$, which can be checked directly by \eqref{equ-ttau} and \eqref{fTA}.

 For (2), the upper bound comes from  \eqref{equ-dB} and \eqref{equ-sS}. For the lower bound, since $\tau_h$ is the maximizer, for $\theta$ as in \eqref{equ-Atheta} and $r_0 > 0$ sufficiently small, we have
\begin{equation*}
\sS(\tau_h) \ge \sS(r_0 \theta) \gtrsim  \theta^T \mathbb{A} \theta + \sumj a_j |\theta| \Big(  r_0 \cot(a_j r_0 |\theta|) +  \mu(a_j |\theta|) \Big)  |\xj|^2  \sim d_\mathbb{B}(g)^2,
\end{equation*}
 by the facts that $r \cot r \sim 1+ r \mu(r) \sim (\frac{r}{\sin r})^2 \sim 1$ near $r=0$, and $1+ r \mu(r) \sim (\frac{r}{\sin r})^2$ for $0<r<\pi$.

To prove (3),  we show $(\pi - r_*) \, \sim \, 1$, where $r_*$ is as in Lemma \ref{Lm43}.  Since Assumption \ref{assumption} implies $\mathfrak{c}_\ell \sim 1$, it suffices to notice that
\begin{equation*}
\, r_* \cot r_* = - \frac{1}{ \mathfrak{c}_\ell} \sum_{j=1}^{\ell-1} \mathfrak{c}_j \,  (a_j r_*) \cot(a_j r_*) \, \ge - \frac{1}{ \mathfrak{c}_\ell } \sum_{j=1}^{\ell-1} \mathfrak{c}_j = \frac{\mathfrak{c}_\ell -1 }{\mathfrak{c}_\ell }. 
\end{equation*}

The proof of (4) is similar to that of \cite[Lemma 5.4 (c)]{BLZ25} and is omitted here.
\end{proof}

\begin{remark}\label{remark-pi}
 Lemma \ref{lem-he} (3) says that $\tau_h$ is bounded away from $\partial B_{\R^m} 
(0,\pi)$. This property is crucial for establishing uniform asymptotic behavior of Poisson kernel; see Proposition \ref{prop-poi}. Recall that in the classical H-type groups, we have stronger result $ |\tau_h| < \frac{\pi}{2}$; see \cite[Lemma 5.4 ($a$)]{BLZ25}. Moreover, Assumption \ref{assumption} is in fact necessary. For example, take $\ell =2$, $0 < a_1 < \frac{1}{2} < 1 = a_2$.
Then $r_*\rightarrow \pi^{-}$ as $\mathfrak{c}_2\rightarrow 0^{+}$, since $\CG(r_*) = 0$ implies 
$
-\frac{\cot r_*}{\cot\left(a_1 \, r_*\right)} = a_1 \, \mathfrak{c}_1 \,
\mathfrak{c}_2^{-1}  \rightarrow + \infty.
$
By \eqref{equ-ttau}, we know $\rho_h \rightarrow r_*^{-}$ as $|t|\rightarrow +\infty$. Thus $\rho_h \rightarrow \pi^{-}$ as $|t|\rightarrow +\infty$ and $\mathfrak{c}_2 \rightarrow 0^{+}$.
\end{remark}

\subsection{The Hessian of $\sS$}\label{sec-hess}

Recall $\sS(\cdot) = \sS_\mathbb{A}(g,h\,; \cdot)$ in \eqref{equ-sS}.  We  compute  the Hessian matrix of $\sS$ at $\tau_h$. By Lemma \ref{Lm43}, it holds that: 
\begin{equation*}
\nabla_{\lambda} \,\sS_\mathbb{A}(x, \MC(\lambda), h \, ; \lambda) = 0, \quad \forall \,\lambda \in B_{\R^m}(0,r_*).
\end{equation*}
Let $\frac{\partial^2 \sS }{\partial \lambda \partial t} $ be the mixed Hessian and $\mathbb{J}$ be as in \eqref{equ-JJ}. The chain rule then gives that
\begin{align*}
\begin{split}
- & \mbox{Hess}_{\lambda}\, \sS_\mathbb{A} (x, \MC(\lambda), h \,; \lambda) 
= 4^{-1} \frac{\partial^2 \sS_\mathbb{A}}{\partial \lambda \partial t} (x, \MC(\lambda), h \, ; \lambda) \, \mathbb{J}(\lambda)\\
& =  \left( \Lambda(|\lambda|) \, \mathbb{I}_m + \lambda \lambda^T  \frac{\Lambda'(|\lambda|)}{|\lambda|} \right) \mathbb{J}(\lambda)  = \mathcal{F}_1(\lambda) \, \mathbb{I}_m  + 8\Lambda(|\lambda|) \mathbb{A}+ \mathcal{F}_2(\lambda) \lambda\lambda^T, 
\end{split}
\end{align*}
where  in the last equality we use $u(r) = -\frac{\varphi(r)}{r \, (1 + r \varphi(r))}$ (cf. \eqref{fCG}), and write 
\begin{align}
\mathcal{F}_1(\lambda) = \Lambda(|\lambda|) \, \TA(|\lambda|), \quad  \mathcal{F}_2(\lambda) = \Lambda(|\lambda|) \, \CG(|\lambda|) \, \TB(\lambda) + \frac{\Lambda'(|\lambda|)}{|\lambda|} \, \TA(\lambda).
\end{align}

The following proposition is the counterpart of \cite[eq. (5.18)-(5.19)]{BLZ25}:

\begin{prop}\label{prop-Hess}
Under Assumption \ref{assumption}, we have uniformly in $ 0 < h \le n^{-\frac{1}{2}} d_{\mathbb{B}} (g)$,
\[
- \mathrm{Hess}_{\lambda}\, \sS (\tau_h) \sim \sS(\tau_h) \, \mathbb{I}_m + \mathbb{A} \quad \mbox{and} \quad \det \left( - \mathrm{Hess}_{\lambda}\, \sS (\tau_h) \right) \lesssim \prod_{l=1}^m \big( d_{\mathbb{B}}(g)^2 + 8 b_l \big). 
\]
\end{prop}

\begin{proof}
   By 
  \cite[Theorem~VI.7.1]{Bhatia97}  and Lemma \ref{lem-he} (2), it suffices to show that:
   \[
(1) \ \sS(\tau_h) \lesssim \mathcal{F}_1(\tau_h) \le \sS(\tau_h); \qquad (2) \ 0\le \mathcal{F}_2(\tau_h)|\tau_h|^2 \lesssim \sS(\tau_h).
\]
For convenience, we may abbreviate $f_1(\rho_h)=f_1$ for any function $f_1$. 

We prove (1) first.  From \eqref{NST}, it follows directly that
\begin{align*}
\mathcal{F}_1(\tau_h) =  \Lambda \, \TA =\Lambda \, \Tg + 4u \Lambda \tau_h^T \mathbb{A} \tau_h &= -\frac{\varphi }{\rho_h} \, \sS(\tau_h) + \frac{\Lambda}{\rho_h} \sumj a_j \mu(a_j \rho_h) \, |\xj|^2 \\
 &\ge  -\frac{\varphi }{\rho_h} \, \sS(\tau_h) \gtrsim  \sS(\tau_h), 
\end{align*}
 by the second equality in \eqref{equ-f'}.  
Moreover, the inequality $\frac{\mu(r)}{r} \left(\frac{\sin r}{r}\right)^2 \le \frac{2}{3}$ (cf. \cite[eq. (5.7)]{BLZ25}) gives $\mathcal{F}_1(\tau_h) \le ( -\frac{\varphi}{\rho_h} + \frac{2}{3} \, \CG ) \,  \sS(\tau_h)$.
It then suffices to prove
$-\frac{\varphi(r)}{r} + \frac{2}{3} \,  \CG(r) \le 1$ whenever $r\in (0,r_*).$
Since $\varphi(r) = r^{-1} (\CG(r)-1 )$, it is equivalent to show $(\frac{2}{3}r^2 -1) \, \mathcal{G}(r) \le r^2 -1$. The case where $r \in [1, r_*)$ can be easily derived by the fact $ 0\le \mathcal{G}(r) \le 1$.
For the case $r\in (0,1)$, we use the inequality
$$\CG(r) \ge r \cot r = \tfrac{r}{\sin r} \cos r \ge (1 + \tfrac{r^2}{6})  (1 - \tfrac{r^2}{2}) \ge (1-r^2 ) (1-\tfrac{2}{3}r^2 )^{-1},$$ where in the second ``$\ge$" we 
use  \cite[eq. (5.7), (3.5)]{BLZ25}. 

To prove (2),  by the fact that $\Lambda(r) \sim 1$ provided $0 \le r \le r_*$, we need to show 
\begin{equation*}
0 \le \rho_h \left( \varphi \left( \Tg + 4 u \tau_h^T \mathbb{A} \tau_h \right) + (1+ \rho_h \varphi ) \left( \Tg' + 4 \tau_h^T \mathbb{A} \tau_h u' \right) \right) \lesssim \sS(\tau_h).
\end{equation*}
 By \eqref{NST}, it remains to establish:
\begin{align}
& 0 \le u \varphi + (1+ \rho_h \varphi)u' \lesssim \frac{1}{\rho_h (1+ \rho_h \varphi)}, \label{equ-F21}\\
& 0 \le \varphi \Tg + (1+ \rho_h \varphi) \Tg' \lesssim \frac{1}{\rho_h (1+ \rho_h \varphi)} \left( h^2 +  \sumj
	\left( \frac{a_j \rho_h}{\sin {(a_j \rho_h)}} \right)^2 \, |x_{(j)}|^2 \right).\label{equ-F22}
\end{align}

First, \eqref{equ-F21} is a consequence of the identity 
$
u \varphi + (1+ \rho_h \varphi)u' 
= \frac{-\rho_h \varphi' + \varphi (1+ \rho_h \varphi)}{\rho_h^2 (1+ \rho_h \varphi)}
$
as well as the inequality
\begin{equation}\label{equ-rho'}
0\le -\rho_h \varphi' + \varphi(1+ \rho_h \varphi) \lesssim \rho_h.
\end{equation}
To verify \eqref{equ-rho'},  by \eqref{equ-f'} we know $-\frac{\varphi(r)}{r}$ is increasing on  $(0, r_*]$. Thus
\[
0 \le  \rho_h^2 \left( - \frac{\varphi(r)}{r} \right)' \Big|_{r=\rho_h} = -\rho_h \varphi' + \varphi \le -\rho_h \varphi' + \varphi(1+ \rho_h \varphi) \le -\rho_h \varphi' \lesssim \rho_h,
\]
as desired.
To prove \eqref{equ-F22}, a direct  calculation and \eqref{fTA} yield 
\begin{align*}
\begin{split}
	& \rho_h \, \Big( \varphi \Tg + (1+ \rho_h \varphi) \Tg' \Big) = (1+ \rho_h \varphi)(\Tg + \rho_h  \Tg') - \Tg \\
	& = \left( \frac{\varphi^2 - \varphi'}{1+ \rho_h \varphi}  -u \right) \left( h^2 + \sumj
	\left( \frac{a_j \rho_h}{\sin {(a_j \rho_h)}} \right)^2 \, |x_{(j)}|^2 \right) + \sumj a_j^2  |\xj|^2 \left( \mu'(a_j \rho_h) -\frac{\mu(a_j \rho_h)}{a_j \rho_h} \right).
\end{split}
\end{align*}
By \eqref{equ-rho'} and the fact $u \ge 0$, we see that $0 \le  \frac{\varphi^2 - \varphi'}{1+ \rho_h \varphi}  - u \le \frac{\varphi^2 - \varphi'}{1+ \rho_h \varphi} \lesssim \frac{1}{1+ \rho_h \varphi}. $ Moreover, note that $\frac{1}{1+ \rho_h \varphi} \ge 1$, and by \eqref{MU},
\begin{equation*}
     \mu'(r) - \frac{\mu(r)}{r}  = r \left( \frac{\mu(r)}{r} \right)' \ge 0, \qquad  \left( \frac{\mu(r)}{r} \right)' \Big|_{r=\rho_h} \lesssim 1.
\end{equation*}
 The proof is then finished.
\end{proof}

\subsection{Asymptotics of the Poisson Kernel}\label{sec-asypoi}

Recall from \eqref{equ-poisson} that 
$$ Q_h(g;\lambda) = \mathbf{V}(\lambda) \left( \Phi(g;\lambda) + h^2 \right)^{-N} .$$ By Lemma \ref{lem-he} (1) and arguing as in \cite[Lemma 3.3]{BLZ25}, we obtain:
\begin{equation}\label{equ-QQ}
Q_h(g;\lambda + \mathrm{i} r \hat{\tau_h} \,) \le \mathbf{V}(\mathrm{i} r \hat{\tau_h} \, ) \, \mathbf{V}( \lambda ) \Lambda(r)^N \Big( \sS(r \hat{\tau_h}) + 4 \Lambda(r) \lal  \Big)^{-N}, \ \ r\in (0,\rho_h].
\end{equation}
Moreover, by Cauchy fundamental theorem and \eqref{equ-poisson},
\begin{equation*}
P_{h}(g) = \frac{2^m \Gamma(N) }{\pi^N} h \int_{\R^m} Q_h(g;\lambda + i \tau_h) \, d\lambda,\quad g\in\GG,\ \  h>0.
\end{equation*}

As an analogue of \cite[Proposition 5.5]{BLZ25}, one can similarly prove the following result using Proposition \ref{prop-Hess}, \eqref{equ-QQ} and Lemma \ref{lem-he} (4). 

\begin{prop}\label{prop-poi}
With Assumption \ref{assumption}, it holds uniformly in $0  < h \le n^{-\frac{1}{2}} d_{\mathbb{B}} (g)$ that
\begin{equation}\label{equ-Ph}
P_h(g) \sim \frac{\Gamma(N)}{\pi^{N}}h\left(\frac{8\pi}{N}\,\sS(\tau_h)\right)^{\frac{m}{2}} Q_h(\mathrm{i} \tau_h)\, \frac{1}{\sqrt{\det (-\mathrm{Hess}_\lambda \, \sS(\tau_h))}}.
\end{equation}
\end{prop}

\subsection{Proof of Theorem \ref{thm-main1}}\label{sec-pf}

The proof of Theorem \ref{thm-main1} is similar to that of \cite[Theorem 5.1]{BLZ25}.
If $n$ large, 
we apply Theorem \ref{thm-voll}, Lemma \ref{lem-he} (2) and Propositions \ref{prop-Hess}-\ref{prop-poi}. 
When $n\lesssim 1$, we can combine the Vitali covering Lemma with Theorem \ref{thm-voll} to get the conclusion.\\ 

\subsection{Results on a Homogeneous Norm $d_\mathbb{G}$}\label{ssec-dk}

 Remark that the following characterization holds for the Kor\'{a}nyi norm on H-type groups
\[
d_{\mathrm{K}}(x,t)^2 = \left(|x|^4 + 16|t|^2\right)^{\frac{1}{2}} = \sup_{\lambda \in B_{\R^m}(0,\pi)} \frac{\sin|\lambda|}{|\lambda|} \left( |x|^2 |\lambda| \cot |\lambda| + 4 t \cdot \lambda \right).
\]
Inspired by this, we introduce its counterpart on $\HH = \G$ by 
\begin{equation}\label{equ-defdH}
d_{\HH}(g)^2 \, : = \sup_{\lambda \in B_{\R^m}(0, \pi)} \Lambda\left( \frac{ |\lambda| }{a_{\ell}} \right) \, \left( \sumj  \left(\frac{a_j}{a_{\ell}} |\lambda| \right) \cot{ \left(\frac{a_j}{a_{\ell}} |\lambda| \right)} \, |\xj|^2 + \frac{4}{a_\ell} \, t \cdot \lambda \right).
\end{equation} 

Let $B_{\HH}(o,r)$ denote the ball centered at $o$ with radius $r>0$ associated with $d_{\HH}$ , and $M_\HH$ the corresponding centered maximal operator associated to $d_{\HH}$. 
Similar to the proof of Theorem \ref{thm-voll} and Theorem \ref{thm-main1}, we can readily establish the following generalization of the known results, obtained in \cite{LQ14}, \cite{BLZ25}, on H-type groups w.r.t. $d_{\mathrm{K}}$.

\begin{theorem}\label{thm-2}
We have
$|B_{\HH}(o,1)| \, \sim  \, |B_{\R^{2n+m}}(0,1)|   \left(   \CCC/ 8 \right)^{\frac{m}{2}}. $
Moreover, if Assumption \ref{assumption} holds, then 
$\| M_{\HH} \|_{L^1 \longrightarrow L^{1,\infty}} \lesssim  \left(   a_\ell^2 \, \CCC^{-1}   \right)^{\, \frac{m}{2}} \, n.$\\
\end{theorem}

\section*{Acknowledgement}
	\addcontentsline{toc}{section}{Acknowledgement}
We thank the anonymous referees for the careful reading of the manuscript and for
giving us several useful suggestions on the presentation of the paper. We also thank Yimeng Chen for her helpful advice on the English expression.

\noindent
\\
\\
{\small\it\\
\noindent
Cheng Bi, Hong-Quan Li}\\
{\small\it   School of Mathematical Sciences,
Fudan University}\\
{\small\it  220 Handan Road,
Shanghai 200433 China}\\
{\small\it E-mail:\;\;\tt cbi21@m.fudan.edu.cn, hongquan\_li@fudan.edu.cn}\\


\small
\begin{thebibliography}{100}
	\bibitem{BGG96}
	R.~Beals, B.~Gaveau, and P.~Greiner, ``The {G}reen function of model step two
	hypoelliptic operators and the analysis of certain tangential {C}auchy
	{R}iemann complexes,'' {\em Adv. Math.}, vol.~121, no.~2, pp.~288--345, 1996.

    \bibitem{Bhatia97}
	R.~Bhatia, {\em Matrix analysis}, vol.~169 of {\em Graduate Texts in
		Mathematics}.
	\newblock Springer-Verlag, New York, 1997.

    \bibitem{BLZ25}
	C.~Bi, H.-Q. Li, and Y.~Zhang, ``Centered {H}ardy-{L}ittlewood maximal
	functions on {H}-type groups revisited,'' {\em Math. Ann.}, vol.~391, no.~3,
	pp.~3765--3797, 2025.

    \bibitem{CCFI}
	O.~Calin, D.-C. Chang, K.~Furutani, and C.~Iwasaki, {\em Heat kernels for
		elliptic and sub-elliptic operators}.
	\newblock Applied and Numerical Harmonic Analysis, Birkh\"{a}user/Springer, New
	York, 2011.
	\newblock Methods and techniques.
	
	\bibitem{EGS18}
	N.~Eldredge, M.~Gordina, and L.~Saloff-Coste, ``Left-invariant geometries on
	{$\rm SU(2)$} are uniformly doubling,'' {\em Geom. Funct. Anal.}, vol.~28,
	no.~5, pp.~1321--1367, 2018.

	\bibitem{EGS24}
	N.~{Eldredge}, M.~{Gordina}, and L.~{Saloff-Coste}, ``{Uniform doubling for
		abelian products with $\operatorname{SU}(2)$},'' {\em arXiv e-prints},
	arXiv:2412.17102, Dec. 2024.

    \bibitem{GR15}
	I.~S. Gradshteyn and I.~M. Ryzhik, {\em Table of integrals, series, and
		products}.
	\newblock Elsevier/Academic Press, Amsterdam, eighth~ed., 2015.
	\newblock Translated from the Russian, Translation edited and with a preface by
	Daniel Zwillinger and Victor Moll, Revised from the seventh edition
	[MR2360010].

    \bibitem{K02}
	S.~G. Krantz and H.~R. Parks, {\em The implicit function theorem}.
	\newblock Birkh\"{a}user Boston, Inc., Boston, MA, 2002.
	\newblock History, theory, and applications.

	\bibitem{Li09}
	H.-Q. Li, ``Fonctions maximales centr\'{e}es de {H}ardy-{L}ittlewood sur les
	groupes de {H}eisenberg,'' {\em Studia Math.}, vol.~191, no.~1, pp.~89--100,
	2009.
	
	\bibitem{Li13}
	H.-Q. Li, ``Remark on ``{M}aximal functions on the unit {$n$}-sphere'' by
	{P}eter {M}. {K}nopf (1987),'' {\em Pacific J. Math.}, vol.~263, no.~1,
	pp.~253--256, 2013.

    \bibitem{Li21}
	H.-Q. Li, ``{The Carnot-Carath{\'e}odory distance on $2$-step groups},'' {\em
		arXiv e-prints}, Dec. 2021.
	
	\bibitem{LL12}
	H.-Q. Li and N.~Lohou\'{e}, ``Fonction maximale centr\'{e}e de
	{H}ardy--{L}ittlewood sur les espaces hyperboliques,'' {\em Ark. Mat.},
	vol.~50, no.~2, pp.~359--378, 2012.
	
	\bibitem{LQ14}
	H.-Q. Li and B.~Qian, ``Centered {H}ardy-{L}ittlewood maximal functions on
	{H}eisenberg type groups,'' {\em Trans. Amer. Math. Soc.}, vol.~366, no.~3,
	pp.~1497--1524, 2014.

    \bibitem{M76}
	J.~Milnor, ``Curvatures of left invariant metrics on {L}ie groups,'' {\em
		Advances in Math.}, vol.~21, no.~3, pp.~293--329, 1976.
	
	\bibitem{NT10}
	A.~Naor and T.~Tao, ``Random martingales and localization of maximal
	inequalities,'' {\em J. Funct. Anal.}, vol.~259, no.~3, pp.~731--779, 2010.

    \bibitem{SS83}
	E.~M. Stein and J.-O. Str\"{o}mberg, ``Behavior of maximal functions in {${\bf
			R}^{n}$} for large {$n$},'' {\em Ark. Mat.}, vol.~21, no.~2, pp.~259--269,
	1983.
	
	\bibitem{VSC92}
	N.~T. Varopoulos, L.~Saloff-Coste, and T.~Coulhon, {\em Analysis and geometry
		on groups}, vol.~100 of {\em Cambridge Tracts in Mathematics}.
	\newblock Cambridge University Press, Cambridge, 1992.
	
\end{thebibliography}
\end{document}